\title{Small scale creation for 2D free boundary Euler equations\\ with surface tension}
\author{Zhongtian Hu\thanks{
                  Department of Mathematics, Duke University, Durham, NC 90320, USA; email: zhongtian.hu@duke.edu}
        \and
        Chenyun Luo\thanks{
                      Department of Mathematics, the Chinese University of Hong Kong, Shatin, NT, Hong Kong; email: cluo@math.cuhk.edu.hk
                  }
                  \and
                  Yao Yao\thanks{Department of Mathematics, National University of Singapore, 119076 Singapore; email: yaoyao@nus.edu.sg}
        }
\documentclass[11pt]{article}
\usepackage[top=1in, bottom=1.in, left=1.in, right=1.in]{geometry}

 \usepackage[bookmarksnumbered, bookmarksopen, colorlinks, citecolor=red, linkcolor=black]{hyperref}
 \usepackage{amsfonts}
\usepackage{graphicx}
 \usepackage{amsmath, amssymb}
 \usepackage{amsthm}
 \usepackage{comment}
 \usepackage{mathtools}
 \usepackage{calc}
 \usepackage{cancel}
 \usepackage{accents}
\newcommand{\ubar}[1]{\underaccent{\bar}{#1}}

\usepackage{accents}

\numberwithin{equation}{section}
\newtheorem{thm}{Theorem}[section]
\newtheorem{lem}[thm]{Lemma}
\newtheorem{prop}[thm]{Proposition}
\newtheorem{rmk}[thm]{Remark}

\usepackage{xcolor}
\definecolor{purple}{rgb}{0.5, 0, 1}
\definecolor{orange}{rgb}{1,.5,0}

\newcommand{\eps}{\varepsilon}

\newcommand{\Z}{\mathbb{Z}}

\newcommand{\T}{\mathbb{T}}

\newcommand{\calD}{\mathcal{D}}

\newcommand{\calH}{\mathcal{H}}

\newcommand{\calO}{\mathcal{O}}

\newcommand{\p}{\partial}

\newcommand{\nab}{\nabla}

 \newcommand{\tDt}{\widetilde{\calD_t}}
 
 \newcommand{\tGamma}{\widetilde{\Gamma}}
 \newcommand{\tNN}{\widetilde{\mathcal{N}}}
 \newcommand{\cnab}{\overline{\nabla}}
 \newcommand{\ctnab}{\widetilde{\cnab}}

\begin{document}
\newpage
\maketitle
%%%%%%%%%%%%%%%%%%%%%%%%%%%
% abstract, keywords, and Subject classification are optional.
%%%%%%%%%%%%%%%%%%%%%%%%%%%
\begin{abstract}
{In this paper, we study the 2D free boundary incompressible Euler equations with surface tension, where the fluid domain is periodic in $x_1$, and has finite depth.}  We construct initial data with a flat free boundary and arbitrarily small velocity, such that the gradient of vorticity grows at least double-exponentially for all times during the lifespan of the associated solution. This work generalizes the celebrated result by Kiselev--{\v{S}}ver{\'a}k \cite{kiselev2014small} to the free boundary setting. The free boundary introduces some major challenges in the proof due to the deformation of the fluid domain and the fact that the velocity field cannot be reconstructed from the vorticity using the Biot-Savart law. We overcome these issues by deriving uniform-in-time control on the free boundary and obtaining pointwise estimates on an approximate Biot-Savart law.
\end{abstract}
MSC codes: 35Q35, 76B45, 
%%%%%%%%%%%%%%%%%%%%%%
% % Here is the start of the Text
%%%%%%%%%%%%%%%%%%%%%%

\section{Introduction}
The 2D incompressible free boundary Euler equations describe the motion of a fluid in two dimensions with a free boundary separating the moving fluid region $\calD_t$ and the vacuum region. 
In the fluid region, the fluid velocity $u(t,x)$ and the pressure $p(t,x)$ satisfy the incompressible Euler equations:
\begin{equation}
    \label{eq:fbeulermodel}
    \begin{cases}
        \p_t u + u\cdot \nabla u + \nabla p = 0,& \text{ in } \calD_t,\\
        \nabla\cdot u = 0,& \text{ in } \calD_t.
    \end{cases}
\end{equation}
We consider the setting where the whole spatial domain is $\mathbb{T}\times\mathbb{R}_+$, where $\mathbb{T} = [-1,1)$ has periodic boundary condition. Assume the fluid domain $\calD_t$ consists of an upper moving boundary $\Gamma_t$ and a fixed flat bottom $\Gamma_b = \mathbb{T}\times\{x_2=0\}$.  Here the free boundary $\Gamma_t$ evolves according to the fluid velocity $u(t,x)$, namely, its normal velocity $V$ is given by
\begin{equation}\label{eq:kinematic}
V = u\cdot \mathcal{N} \quad \text{ on } \Gamma_t, 
\end{equation}
where $\mathcal{N}$ is the outward unit normal to $\Gamma_t$. Throughout this paper, we assume the presence of surface tension, i.e., the pressure on the free boundary obeys
\begin{equation}\label{bdry_t}
p=\sigma\calH \quad \text{ on } \Gamma_t,
\end{equation}
where $\sigma>0$ is the surface tension constant, and $\calH$ is the mean curvature of the free boundary. %\textcolor{orange}{(I changed mean curvature to curvature. I think mean curvature is only for higher dimensions?)} 
On the fixed boundary, we impose the no-flow boundary condition
\begin{equation}\label{bdry_b}
u\cdot n=0\quad \text{on}\,\,\Gamma_b,
\end{equation}
where $n=(0,-1)$ is the outward unit normal to $\Gamma_b$.  For simplicity, let the initial free boundary be a straight line $\Gamma_0 = \mathbb{T}\times \{x_2=2\}$, so the initial fluid domain is
\begin{equation}\label{reference domain}
\calD_0= \mathbb{T} \times (0, 2).
\end{equation}

The system \eqref{eq:fbeulermodel}--\eqref{bdry_b} is also referred to as the 2D capillary water wave system. 
This system has been under very active investigation for the past two decades. %For simplicity of the presentation, we will focus on the literature which includes surface tension effects.
The local well-posedness for the free-boundary Euler equations with surface tension is well-known, which can be found in \cite{alazard2011water, castro2015well, coutand2007well, disconzi2019priori, disconzi2019lagrangian, MR2172858, SZ1, SZ2, SZ3}. 
Unlike the Euler equations in a fixed domain, the local well-posedness for free-boundary Euler equations does not come directly from the a priori estimate since the linearized equations lose certain symmetry on the moving boundary. This issue is resolved by introducing carefully designed approximate equations that are asymptotically consistent with the a priori estimate.
 In addition, for certain large initial data, it is known that the solution to the water wave system with or without surface tension can form a splash singularity in finite time; see \cite{castro2013finite, castro2012finite, castro2012splash,   coutand2014finite}.

Beyond local well-posedness, a natural question is whether solutions with small initial data stay small {for a longer period of time}.   For  \textit{irrotational} $u_0$ (i.e. $\nabla\times u_0 = 0$) in a domain with infinite depth, 
%a positive answer was given independently by Ifrim--Tataru \cite{ifrim2017lifespan} (for either periodic or asymptotically flat free boundary) and Ionescu--Pusateri \cite{ionescu2018global} (for asymptotically flat free boundary), 
{a positive answer was given independently by Ifrim--Tataru \cite{ifrim2017lifespan} and Ionescu--Pusateri \cite{ionescu2018global} for an asymptotically flat free boundary,} where they showed that small initial data leads to a global-in-time small solution.  
{As for the case with a periodic free boundary, Ifrim--Tataru \cite{ifrim2017lifespan} proved that small data solutions of the infinite depth water waves in two space dimensions have at least cubic lifespan. In addition, Berti--Feola--Franzoi \cite{berti2021quadratic} proved a similar result but with a finite bottom. See also Berti--Delort \cite{berti2018almost}, in which the almost global existence of 2D gravity-capillary water waves is established, provided that additional symmetry conditions are imposed on the small initial data.}
The key strategy in the aforementioned works is to reduce the system\eqref{eq:fbeulermodel}--\eqref{bdry_t} to a new system of equations defined on the moving boundary $\Gamma_t$, owing to the fact that $u$ is both divergence- and curl-free. {See also Deng--Ionescu--Pausader--Pusateri \cite{deng2017global} for global-in-time irrotational solutions of the gravity-capillary water-wave system in 3D.} However, for \textit{rotational} $u_0$ it is unknown whether solutions with small initial data always remain small for all times. 
 
 The goal of this work is to give a negative answer to this question in the finite-depth case - namely, we construct smooth initial data with a flat free boundary and arbitrarily small velocity, where $\|\nabla\omega(t)\|_{L^\infty}$ grows at least double-exponentially for all times during the lifespan of the solution. 

For 2D Euler equations in a disk, such double-exponential growth of  $\|\nab\omega(t)\|_{L^\infty}$ was constructed in a celebrated result by 
Kiselev--{\v{S}}ver{\'a}k \cite{kiselev2014small}. Similar ideas were applied to the torus $\mathbb{T}^2$ by Zlato\v{s} \cite{zlatovs2015exponential} to obtain exponential growth of vorticity gradient, and applied to smooth domains with an axis of symmetry by Xu \cite{xu2016fast}. In this paper, we aim to extend the construction of \cite{kiselev2014small} to the free boundary setting.  Our main result is as follows, which is stated for the $\sigma=1$ case for simplicity: 

 \begin{thm}\label{thm:smallscale}
Consider the 2D free boundary Euler equations \eqref{eq:fbeulermodel}--\eqref{bdry_b} with $\sigma=1$, whose initial domain $\calD_0$ is given by \eqref{reference domain}. There exists a smooth velocity field $v_0 \in C^\infty(\calD_0)$ and universal constants $\eps_0, c_1, c_2 > 0$, such that for any $\eps\in (0,\eps_0)$, the solution\footnote{Here and throughout, a \textit{solution} always means a $H^s$--solution, for some fixed $s\geq 4$. Since our initial data is smooth, the local existence of such a solution is guaranteed by \cite{coutand2007well, MR2172858}. } to \eqref{eq:fbeulermodel}-\eqref{bdry_b} with initial velocity $u_0 := \eps v_0$ satisfies the following for its vorticity $\omega:=\partial_1 u_2-\partial_2 u_1$:
 \begin{equation}
\label{est:smallscale}
\|\nabla \omega(t,\cdot)\|_{L^\infty(\calD_t)}\ge \eps \exp(c_1 \exp(c_2 \eps t)) \quad \text{for all}\,\, t\in [0,T),
\end{equation}
% \begin{equation}
%\label{est:smallscale}
%\frac{\|\nabla \omega(t,\cdot)\|_{L^\infty(\calD_t)}}{\|\omega_0\|_{L^\infty(\calD_0)}} \ge \left(\frac{\|\nabla \omega_0\|_{L^\infty(\calD_0)}}{\|\omega_0\|_{L^\infty(\calD_0)}}\right)^{c(\sigma)\exp(c(\sigma)\|\omega_0\|_{L^\infty(\calD_0)}t)} \quad \text{for each}\,\, t\in [0,T),
%\end{equation}
where $T$ is the lifespan of the solution. 
\end{thm}

\begin{rmk}
(1) In other words, we have constructed smooth small initial data of size $\eps\ll 1$, such that $\|u(t)\|_{W^{2,\infty}}$ grows to order one by time $O(\eps^{-1} \ln \ln \eps^{-1})$, unless a singularity occurs before this time. {That is, we have demonstrated nonlinear instability for a class of \textbf{rotational} initial data in their respective lifespans,
%Thus small initial data is lack of the stability in a long time (up to times of order $\epsilon^{-2}$, say), 
%will not lead to a small solution for all times, 
which is a sharp contrast to the irrotational case %with infinite depth
\cite{berti2018almost, berti2021quadratic, deng2017global, ifrim2017lifespan, ionescu2018global}}. This result in 2D can be readily extended to the periodic 3D setting, by setting $u_0$ independent of the $x_3$ variable. 

(2) Theorem~\ref{thm:smallscale} can be easily generalized to all $\sigma>0$ (with $\eps_0, c_1, c_2$ depending on $\sigma$ now). A simple scaling argument shows that if $(u(t,\cdot), \calD_t)$ is a solution to \eqref{eq:fbeulermodel}--\eqref{bdry_b} with $\sigma=1$, then $(\sqrt{\sigma} u( \sqrt{\sigma} t,\cdot), \calD_{\sqrt{\sigma}t})$ solves \eqref{eq:fbeulermodel}--\eqref{bdry_b} for a given $\sigma>0$.
\end{rmk}

 \begin{figure}[thbp]
 \begin{center}
 \includegraphics[scale=0.8]{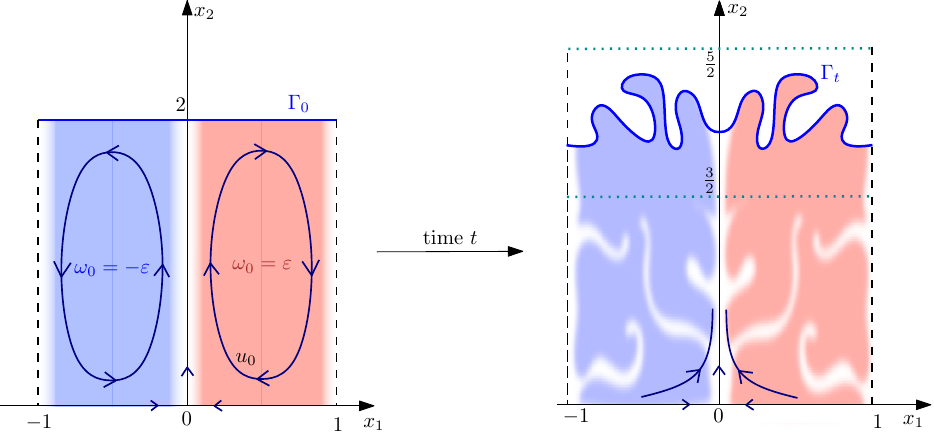}
 \end{center}
 \caption{\label{fig1} Illustrations of our initial data $u_0,\Gamma_0$ and its evolution after time $t$. Here the red and blue colors represent positive and negative vorticity respectively. As we show in Proposition~\ref{prop:bdryctrl}, for small initial velocity, the free boundary $\Gamma_t$ will be confined within $\frac32 < x_2 < \frac52$ for all time during the lifespan of a solution. This allows us to estimate $u(t,\cdot)$ near the point $(0,0)$.}
 \end{figure}

To prove Theorem~\ref{thm:smallscale}, a natural starting point is to enforce the same symmetry as \cite{kiselev2014small}, with $u_{01}$ odd-in-$x_1$ and $u_{02}$ even-in-$x_1$ respectively. One can easily check that such symmetry holds for all times (so vorticity remains odd-in-$x_1$ for all times). The proof is standard and we include it in Lemma~\ref{lem: symmetry} for the sake of completeness. In addition, if we set $\omega_0=1$ at most of the points in the right half of $\calD_0$ (except a small measure, since $\omega_0$ has to smoothly transition to 0 at $x_1=0$ and $x_1=1$ due to its oddness), one can check that such property also holds for $\calD_t$, since in the free boundary setting the vorticity is also preserved along the characteristics. 

However, despite these similarities, one faces two major challenges to adapt the proof of \cite{kiselev2014small} to the free boundary setting:

 The first issue is the deformation of the domain. Since $\omega_0\not\equiv 0$ and $\calD_t$ is evolving in time, it may deform a lot from the initial domain $\calD_0$. In general, the free boundary $\Gamma_t$ might get very close to the origin, and the nonlinear coupling between $\Gamma_t$'s evolution and the velocity field in the bulk of the fluid could destroy the small-scale creation mechanism near the origin in \cite{kiselev2014small}. That being said, we show that this can never happen at any time for small initial data. This is because the free boundary Euler equation with surface tension is known to have a conserved energy $E(t)=K(t)+\sigma L(t)$, where $K(t)$ is the kinetic energy and $L(t)$ is the length of the free boundary. (The energy conservation was shown in \cite{MR2172858, SZ1}, and we derive it in Proposition~\ref{prop: conserved} for the sake of completeness.) Using this conserved energy, we make the simple but important observation that a flat initial free boundary and a small initial kinetic energy guarantees that the free boundary $\Gamma_t$ always stays close to $\Gamma_0$, thus can never get close to the origin -- see Proposition~\ref{prop:bdryctrl} for a precise statement, and see Figure~\ref{fig1} for an illustration.

A more serious problem is the lack of Biot-Savart law in the free boundary setting. Recall that for a \emph{fixed domain} $D$, given the vorticity $\omega(t,\cdot)$ in $D$ at any moment, the velocity field $u(t,\cdot)$ is uniquely determined by the Biot-Savart law $u = \nabla^\perp\varphi$, where the stream function $\varphi$ solves the elliptic equation $\Delta\varphi = \omega$ in $D$ with $\varphi=0$ on $\partial D$. This Biot-Savart law was crucial in \cite{kiselev2014small} to derive pointwise estimates of $u(t, \cdot)$. 
In contrast, in the \emph{free boundary} setting, even with $\omega(t,\cdot)$ and $\calD_t$ given at some time $t$, it is not sufficient to uniquely determine $u(t,\cdot)$ -- one also needs to know the normal velocity of the free boundary to determine $u(t,\cdot)$ in the fluid domain. To overcome this challenge, we show that $u(t,\cdot)$ can still be somewhat determined by $\omega(t,\cdot)$ by an \emph{approximate Biot-Savart law} in Section~\ref{sec_approx}, which contains an error term that remains regular and small for all times near the origin. This allows us to obtain a pointwise estimate of $u$ similar to the key lemma in \cite[Lemma 3.1]{kiselev2014small}, leading to the double-exponential growth of $\|\nabla\omega(t)\|_{L^\infty(\calD_t)}$.

\medskip

%In the present paper, we extend the result of \cite{kiselev2014small} to 2D free-boundary Euler equations. 
%
%
%
%
% \begin{thm}\label{thm:smallscale}
% For each fixed $\sigma>0$, there exists smooth initial data $u_0$, such that the vorticity $\omega(t,x) := \nabla^\perp\cdot u(t,x)$ of the corresponding solution $u(t,x)$ for the 2D free-boundary Euler equations \eqref{eq:fbeulermodel}, \eqref{eq:ibc} satisfies
% \item \begin{equation}
%\label{est:smallscale}
%\frac{\|\nabla \omega(t,\cdot)\|_{L^\infty(\calD_t)}}{\|\omega_0\|_{L^\infty(\calD_0)}} \ge \left(\frac{\|\nabla \omega_0\|_{L^\infty(\calD_0)}}{\|\omega_0\|_{L^\infty(\calD_0)}}\right)^{c(\sigma)\exp(c(\sigma)\|\omega_0\|_{L^\infty(\calD_0)}t)} \quad \text{for each}\,\, t\in [0,T),
%\end{equation}
%where $T$ is the lifespan of the solution, and $c(\sigma)>0$ is a constant depends only on $\sigma$. 
% \end{thm}

 \subsection*{Notations}
 \begin{itemize}
% \item We denote by $u$, $p$, $\omega$ the fluid's velocity, pressure, and vorticity, receptively. Also, $\eta$ is the flow map of $u$. 
% \item The initial fluid domain $\calD_0 = \mathbb{T}\times (0,2)$, and for each $t> 0$, we use $\calD_t=\eta(t, \calD_0)$ to denote the moving fluid domain at time $t$, which consists of a moving top boundary $\Gamma_t$ and a rigid flat bottom $\Gamma_b$.  
 %\item We denote by $T>0$ the lifespan of the smooth solution defined in Theorem \ref{thm:smallscale}. 
% \item We shall abbreviate $L^2(\calD_0)$ and $L^2(\calD_t)$ in our norms and typically write only $L^2$ if there is no ambiguity on where the target functions live. For example, in \eqref{est:smallscale}, the quantities $\|\nab \omega(t, \cdot)\|_{L^2(\calD_t)}$ and $\|\nab \omega_0\|_{L^2(\calD_0)}$  will be shortened to $\|\nab \omega\|_{L^2}$ and $\|\nab \omega_0\|_{L^2}$, respectively. 
 \item Let $B_r$ be the open disk centered at the origin with radius $r$. In Section 3, we define $\Omega := \mathbb{T}\times[0,1]$. We also define $\Omega^+$ and $\calD_t^+$ as the right half of $\Omega$ and $\calD_t$, i.e. $\Omega^+ := [0,1] \times[0,1]$ and $\calD_t^+ := \calD_t \cap \{x_1 \in [0,1]\}$.
 \item We denote by $C$ universal constants whose values may change from line to line. Any constants with subscripts, such as $C_i$, stay fixed once they are chosen.
 \end{itemize}
 
 \subsection*{Acknowledgements}
 
 CL is supported by the Hong Kong RGC grant No. CUHK--24304621 and CUHK--14302922. YY is partially supported by the NUS startup grant, MOE Tier 1 grant A-0008491-00-00, and the Asian Young Scientist Fellowship. ZH {acknowleges  partial support of the NSF-DMS grants 2006372 and 2306726}; he also thanks the hospitality of the Chinese University of Hong Kong. The authors thank Tarek Elgindi for suggesting this problem, and Alexander Kiselev for helpful discussions. {Finally, we thank the anonymous referees for helpful comments which improved the presentation of this paper.}

\section{Preliminary results}
In this section, we collect a few preliminary results on the free boundary Euler equations with surface tension.  In Section~\ref{subsect:symmetry}, We first demonstrate a symmetry to which 2D free boundary Euler equations conform. Such symmetry corresponds to the odd-in-$x_1$ symmetry of vorticity in the fixed boundary case \cite{kiselev2014small}, and is crucial to our construction. In Section~\ref{subsect:conserved}, we show the conservation of vorticity and an energy balance involving the bulk kinetic energy as well as the length of the free boundary.

\subsection{Symmetry in  2D free boundary Euler equations}\label{subsect:symmetry}
To begin with, we discuss some symmetry properties of the 2D free boundary Euler equations. For the 2D Euler equation in fixed domains, the conservation of odd-in-$x_1$ symmetry in vorticity is crucial in the proof of small scale formations, as seen in \cite{kiselev2014small, xu2016fast, zlatovs2015exponential}.
Below we show that a similar symmetry is also preserved for free boundary Euler equations; the difference is that we state the symmetry assumptions in terms of the velocity rather than the vorticity, since for the free boundary Euler equation one cannot uniquely determine the velocity using the vorticity at a given moment due to the kinematic boundary condition \eqref{eq:kinematic}.

\begin{lem}\label{lem: symmetry}
Let $(u_0, \calD_0)$ be the initial data of \eqref{eq:fbeulermodel}--\eqref{bdry_b}, where $\calD_0$ is given by \eqref{reference domain} and $u_0 = (u_{01}, u_{02})$ satisfies 
\begin{equation}\label{symmetry t=0}
u_{01} (-x_1, x_2) = -u_{01} (x_1, x_2), \quad u_{02} (-x_1, x_2) =u_{02} (x_1, x_2). 
\end{equation}
Then for all time during the lifespan of a solution, the solution $(u, \calD_t)$ satisfies the same symmetry, i.e.
\begin{equation}\label{symmetry u}
-u_1 (t, -x_1, x_2) = u_1 (t, x_1, x_2), \quad u_2 (t, -x_1, x_2) = u_2 (t, x_1, x_2),
\end{equation}
and the moving fluid domain $\calD_t$ remains even in $x_1$, i.e.,
\begin{align}
\calD_t = \tDt:=\{(-x_1, x_2) : (x_1, x_2)\in \calD_t\}.
\end{align}
\end{lem}

\begin{rmk}
As a direct consequence of \eqref{symmetry u}, we know the vorticity $\omega(t,x) = \nabla^\perp \cdot u(t,x)$ stays odd in $x_1$ for all time during the lifespan of a solution.
\end{rmk}
\begin{proof}
First, setting 
\begin{align}
v (t, x_1, x_2) = (v_1 (t, x_1, x_2), v_2 (t, x_1, x_2)) &= (-u_1(t, -x_1, x_2), u_2(t, -x_1, x_2)),\label{v}\\
q (t, x_1, x_2) &= p(t, -x_1, x_2), \label{q}
\end{align}
it suffices to show that $(v, q, \tDt)$ also verifies the system \eqref{eq:fbeulermodel}--\eqref{bdry_b} due to uniqueness of solution. Fixing $(x_1,x_2) \in \tDt$, a direct computation shows that
\begin{align*}
(\p_t v_1 + v \cdot \nab v_1 + \p_1 q)|_{(t, x_1, x_2)} &= - (\p_t u_1 + u \cdot \nab u_1 +\p_1 p) |_{(t, -x_1, x_2)},\\
(\p_t v_2 + v \cdot \nab v_2 + \p_2 q)|_{(t, x_1, x_2)} &=  (\p_t u_2 + u \cdot \nab u_2 +\p_2 p) |_{(t, -x_1, x_2)},
\end{align*}
which implies that $\p_t v + v\cdot \nab v + \nab q = 0$ in $\tDt$. Similarly, we have 
$$
\nab \cdot v|_{(t, x_1, x_2)} = \nab \cdot u|_{(t, -x_1, x_2)} = 0.
$$
Second, we need to check the boundary conditions.
Since $\p\tDt= \tGamma_t\cup \tGamma_b$, where
\begin{align}\label{symmetry Gamma_t}
\tGamma_t = \{ (x_1, x_2): (-x_1, x_2)\in \Gamma_t\},
\end{align} 
and 
 $\tGamma_b = \Gamma_b$, then it is straightforward to check that  $v\cdot n = 0$ on  $\tGamma_b$. 
 
Furthermore, denoting by $\tNN=(\tNN_1, \tNN_2)$ the outward unit normal to $\tGamma_t$, we infer from \eqref{symmetry Gamma_t} that
\begin{equation}\label{symm N}
\tNN_1 (t, x_1, x_2) = -\mathcal{N}_1(t, -x_1, x_2), \quad \tNN_2 (t, x_1, x_2) = \mathcal{N}_2(t, -x_1, x_2).
\end{equation}
This yields $v\cdot \tNN |_{(t,x_1,x_2)} = u\cdot \mathcal{N}|_{(t, -x_1, x_2)}$.

Finally, we define $\widetilde{\mathcal{H}}$ to be the mean curvature of $\tGamma_t$. By definition, 
$\mathcal{H} = \cnab\cdot \mathcal{N}, $ where $\cnab$ is the spatial derivative tangent to $\Gamma_t$, whose components read 
\begin{align*}
\cnab_j = \nab_j - \mathcal{N}_j (\mathcal{N}\cdot \nab),\quad j=1,2.
\end{align*} 
This implies 
$\widetilde{\mathcal{H}} = \ctnab \cdot \tNN,
$
where $\ctnab_j = \nab_j - \tNN_j (\tNN\cdot \nab)$. Then, in light of \eqref{symm N}, a direct computation yields that for any $(x_1,x_2) \in \Gamma_t$,
\begin{equation}
\widetilde{\mathcal{H}}(t, x_1, x_2)=\ctnab\cdot \tNN|_{(t, x_1, x_2)}=\cnab \cdot \mathcal{N}|_{(t, -x_1, x_2)} = \mathcal{H}(t, -x_1, x_2).
\end{equation}
Thanks to \eqref{q}, we have
\begin{equation*}
q =\sigma \widetilde{\mathcal{H}},\quad \text{ on }\tGamma_t.
\end{equation*}
This concludes the proof. 
\end{proof}

\subsection{Conservation of vorticity and a conserved $L^2$-energy}\label{subsect:conserved}
In this subsection, we aim to show two conserved quantities satisfied by the free-boundary Euler equations \eqref{eq:fbeulermodel} on both vorticity and velocity sides. The first result below shows that, identical to the classical fixed-boundary Euler equations, any $L^p$ norm of the vorticity is conserved. 
\begin{prop}
\label{prop:vortconserv}
For any $1\leq p \leq \infty$, we have $\|\omega(t,\cdot)\|_{L^p(\calD_t)} = \|\omega_0\|_{L^p(\calD_0)}$ for all times during the lifespan of the solution.
\end{prop}
\begin{proof}
By applying the operator $\nabla^\perp\cdot$ to the velocity equation in \eqref{eq:fbeulermodel} and using the divergence-free property of $u$, $\omega$ satisfies the following transport equation
$$
\p_t \omega + u\cdot \nabla \omega = 0 \quad\text{ in } \calD_t.
$$
This vorticity equation together with the divergence-free property of $u$ yields the result.
\end{proof}

The following proposition shows that free boundary Euler equations with surface tension have a conserved $L^2$-energy. It plays a pivotal role in quantifying the constraining effect of the surface tension on the behavior of the free boundary, as we will see in Section~\ref{subsec_uniform}.
\begin{prop}\label{prop: conserved}
Let
\begin{align}
E(t) := K(t) + \sigma L(t), 
\end{align} 
where
\[K(t) := \frac{1}{2}\int_{\calD_t} |u(t,x)|^2\,dx \quad\text{ and }\quad L(t):=\int_{\Gamma_t} dS_t
\]
are the kinetic energy of the fluid and the length of $\Gamma_t$ respectively. Then 
\begin{align}\label{conserved energy}
E(t) = E(0)
\end{align}
for all times during the lifespan of the solution.
\end{prop}
\begin{proof}
We will verify the identity \eqref{conserved energy} by direct computation. We start from
$$
\frac{d}{dt} K(t) = \int_{\calD_t} (\p_t u +u\cdot \nab u) \cdot u\, dx = -\int_{\calD_t} \nab p \cdot u\,dx,
$$
and then apply the divergence theorem and the boundary conditions to obtain
$$
-\int_{\calD_t} \nab p \cdot u\,dx = -\int_{\Gamma_t} p (u\cdot \mathcal{N})\,dS_t - \int_{\Gamma_b} p \underbrace{(u\cdot n)}_{=0}\,dx_1 + \int_{\calD_t} p\underbrace{(\nab\cdot u)}_{=0}\,dx = -\int_{\Gamma_t} p (u\cdot \mathcal{N})\,dS_t.
$$
 Now,  invoking the boundary conditions $p=\sigma\mathcal{H}$, and $u\cdot \mathcal{N} = V$ on $\Gamma_t$, we have
\begin{align}\label{boundary pun}
-\int_{\Gamma_t} p (u\cdot \mathcal{N})\,dS_t = -\int_{\Gamma_t} \sigma\mathcal{H} V \,dS_t. 
\end{align}
On the other hand, since $\frac{d}{dt}\int_{\Gamma_t} dS_t = \int_{\Gamma_t} \mathcal{H}(u\cdot \mathcal{N})\,dS_t$ (whose proof can be found in \cite[Chapter 4]{ecker2012regularity}), we have
$$
\frac{d}{dt} L(t) =  \int_{\Gamma_t}  \mathcal{H}V\,dS_t.
$$
Combining this with \eqref{boundary pun}, we arrive at
$$
\frac{d}{dt} E(t) = \frac{d}{dt} (K(t) + \sigma L(t))=0,
$$
finishing the proof.
\end{proof}

\section{Uniform-in-time estimates for the free boundary problem}

In this section, we obtain some uniform-in-time estimates of the free boundary and the velocity field, which are at the heart of the proof of the main theorem of the paper.

\subsection{Uniform-in-time control of the free boundary and kinetic energy}
\label{subsec_uniform}
The following result shows that if the initial free boundary is flat and the initial kinetic energy is sufficiently small, the free boundary $\Gamma_t$ stays constrained in a small neighborhood around the initial profile $\Gamma_0$ for all times, and the kinetic energy at time $t$ always stays below the initial kinetic energy.
\begin{prop}
\label{prop:bdryctrl}
Let $\sigma>0$. Consider the solution to the system \eqref{eq:fbeulermodel}--\eqref{bdry_b} with initial fluid domain $\calD_0$ given by \eqref{reference domain}, where the initial velocity $u_0$ is smooth and has a small kinetic energy $K(0) \leq \frac{\sigma}{20}$. 
Then we have
\begin{align}\label{Gamma_t strip}
\Gamma_t \subset \mathbb{T}\times \left(\frac{3}{2}, \frac{5}{2}\right) \quad \text{for all}\,\, t\in[0,T)
\end{align}
and
\begin{equation}
\label{Kt_ineq}
K(t) \leq K(0) \quad \text{for all}\,\, t\in[0,T),
\end{equation}
where $T>0$ is the lifespan of the solution. 
\end{prop}
\begin{proof}
In light of \eqref{conserved energy} in Proposition~\ref{prop: conserved}, we obtain
\begin{equation}\label{energy_eq}
K(t) + \sigma L(t) = K(0) + \sigma L(0)
\end{equation}
for all times during the lifespan of the solution.  
Since $K(t)\geq 0$, we have
\begin{align}\label{estimate l(t)}
L(t) = L(0)+ \frac{K(0) - K(t)}{\sigma} \leq L(0) + \frac{K(0)}{\sigma} \leq 2.05,
\end{align}
 where the last inequality follows from the assumptions \eqref{reference domain} (so $L(0)=2$) and $K(0) \leq \frac{\sigma}{20}$.

  Also, we deduce from the incompressibility that $\calD_t$ has the same area as $\calD_0$, so during the lifespan of the solution, $\Gamma_t$ must intersect with $\Gamma_0=\mathbb{T}\times\{2\}$ at least once. In addition, $\Gamma_t$ is a closed curve in $\mathbb{T}\times\mathbb{R}_+$, and its projection onto the $x_1$ axis is the whole set $\mathbb{T}$. 
  
  For any closed curve satisfying the properties above, if it intersects either $\mathbb{T}\times\{\frac32\}$ or  $\mathbb{T}\times\{\frac52\}$, an elementary computation shows that it must have length at least $2\sqrt{1+(\frac12)^2}=\sqrt{5}\approx 2.236$. Since the length of $\Gamma_t$ stays below 2.05 for all times due to \eqref{estimate l(t)}, we conclude that $\Gamma_t$ must be contained in $\mathbb{T}\times \left(\frac{3}{2}, \frac{5}{2}\right)$ for all times, which proves \eqref{Gamma_t strip}.
  
To show \eqref{Kt_ineq}, note that \eqref{energy_eq} gives $K(t) = K(0) + \sigma(L(0)-L(t))$, where $L(0)=2$. During the lifespan of the solution, the projection of $\Gamma_t$ onto the $x_1$ axis is the whole set $\mathbb{T}=[-1,1)$, thus $\Gamma_t$ has length at least 2. This yields $L(t)\geq 2= L(0)$, thus $K(t)\leq K(0)$.
\end{proof}

\subsection{Error estimates of an approximate Biot-Savart law}
\label{sec_approx}

As we have described in the introduction, a major issue in obtaining pointwise velocity estimates in the free boundary setting is the lack of Biot-Savart law, namely, to determine $u(t,\cdot)$, it is not sufficient to know $\omega(t,\cdot)$ and $\calD_t$. To overcome this challenge, we introduce an ``approximate Biot-Savart law'' which only uses the information of $\omega(t,\cdot)$ in the set $\Omega:= \T \times [0,1]$, which leads to an approximate velocity field $U(t,\cdot)$ in $\Omega$. We will then use the uniform-in-time estimates in Proposition~\ref{prop:bdryctrl} to obtain a precise estimate on the error between the actual velocity $u$ and the approximate velocity field $U$ -- it turns out the error is quite regular and small near the origin.

Recall the notations $\Omega:= \T \times [0,1]$,  and $B_r$ as the open disk centered at the origin with radius $r$. We emphasize that as long as the initial kinetic energy is small, we have $\Omega \subset \calD_t$ for all times during the lifespan of the solution due to Proposition~\ref{prop:bdryctrl}.

For any $t\geq 0$ during the lifespan of the solution, we define an \emph{approximate velocity field} $U(t,\cdot): \Omega\to\mathbb{R}^2$ as
\begin{equation}\label{def_U}
U(t,\cdot) := \nabla^\perp \Phi(t,\cdot)\quad\text{ in }\Omega,
\end{equation}
where $\Phi(t,\cdot)$ solves the following elliptic equation at the fixed time $t$:
\begin{equation}
\label{eq:streamfnc}
\begin{cases}
\Delta \Phi(t,\cdot) = \omega (t,\cdot) &\text{ in }\Omega\\
\Phi(t,\cdot) = 0 &\text{ on }\p \Omega,
\end{cases}
\end{equation}
where $\omega(t,\cdot)=\nabla^\perp \cdot u(t,\cdot)$ is the vorticity of the solution $u(t,\cdot)$.

Note that $U(t,\cdot)$ is uniquely determined by $\omega(t,\cdot)|_\Omega$ using the usual Biot-Savart law for 2D Euler equation in the fixed domain $\Omega$, hence the name ``approximate Biot-Savart law''. To estimate the error between $U$ and the actual velocity field $u(t,\cdot)|_\Omega$ restricted to $\Omega$ (note that $u|_\Omega$ is well defined since $\Omega \subset\calD_t$ for all times by Proposition 3.1), we define the error  $e(t,\cdot): \Omega\to\mathbb{R}^2$ as
\begin{equation}\label{def_e}
e(t,\cdot) := u(t,\cdot)|_\Omega - U(t,\cdot).
\end{equation}

The following proposition plays a key role in our proof of small scale creation. It says that the error $e$ is very regular in $B_{1/2}\cap \Omega$, and $\nabla e$ is pointwise bounded above by $C\sqrt{K(0)}$ for all times. (In fact, the same estimate holds for any higher derivative of $e$, at the expense of having a larger $C$ -- but controlling the first derivative of $e$ is sufficient for us.)

\begin{prop}\label{prop_e}
Let $\sigma>0$. Consider the solution to the system \eqref{eq:fbeulermodel}--\eqref{bdry_b} with initial fluid domain $\calD_0$ given by \eqref{reference domain}, where the initial velocity $u_0$ is smooth and has small kinetic energy $K(0) \leq \frac{\sigma}{20}$. 
During the lifespan of the solution, let $U(t,\cdot)$ and $e(t,\cdot)$ be defined as in \eqref{def_U} and \eqref{def_e} respectively. 

Then $e(t,\cdot)$ is smooth in $B_{1/2}\cap \Omega$ up to its boundary, and there exists a universal constant $C$ such that
\begin{equation}
\label{est_e}
\|\nabla e(t,\cdot)\|_{L^\infty(B_{1/2}\cap \Omega))} \leq C \sqrt{K(0)}.
\end{equation}

\end{prop}

\begin{proof}Let us fix any time $t\geq 0$ during the lifespan of a solution. In the following, all functions are at this frozen time $t$, so for notational simplicity, we will omit their $t$ dependence.

First note that $e$ is divergence-free since both $U$ and $u$ are divergence-free in $\Omega$. By Helmholtz decomposition, there exists a stream function $F:\Omega\to\mathbb{R}$ such that 
\[e = \nabla^\perp F.
\]  
Moreover, $e$ is also  irrotational in $\Omega$:  applying $\nabla^\perp\cdot$ on both sides of \eqref{def_e} (and using the definition of $U$ in \eqref{def_U}--\eqref{eq:streamfnc}), we have
$$
\nabla^\perp \cdot e = \nabla^\perp \cdot u - \nabla^\perp\cdot\nabla^\perp \Phi = \omega - \omega = 0\quad\text{ in } \Omega.
$$
This leads to
\[
\Delta F = \nabla^\perp \cdot e = 0 \quad\text{ in } \Omega.
\] 

In addition, on $\Gamma_b = \mathbb{T}\times\{0\}$, we have $U \cdot n = \nabla^\perp \Phi\cdot n=0$ (from the boundary condition $\Phi=0$ on $\Gamma_b$) and $u\cdot n=0$ by \eqref{bdry_b}. Thus we must also have 
\begin{equation}\label{en}
\nabla^\perp F\cdot n = e \cdot n = (u-U)\cdot n = 0\quad \text{ on }\Gamma_b.
\end{equation} This implies $F=\text{const}$ on $\Gamma_b$, and by adding a constant to $F$ we have $F(t,\cdot)=0$ on $\Gamma_b$ without loss of generality. 
Combining the above, we have shown that there exists $F(t,\cdot):\Omega\to\mathbb{R}$ such that $e(t,\cdot) = \nabla^\perp F(t,\cdot)$, where $F$ satisfies
\begin{equation}
\label{eq:remainderstream}
\begin{cases}
\Delta F(t,\cdot)= 0& \text{ in }\Omega,\\
F(t,\cdot) = 0& \text{ on }\Gamma_b.
\end{cases}
\end{equation}

To show the regularity estimate \eqref{est_e}, we first prove a bound for $\|F\|_{H^1(\Omega)}$. We observe the following orthogonality property between $\nabla\Phi$ and $\nabla F$ in $\Omega$:
\begin{align*}
\int_\Omega \nabla\Phi \cdot \nabla F dx &= -\int_\Omega \Phi \underbrace{\Delta F}_{=0} dx + \int_{\p \Omega} \underbrace{\Phi}_{=0} (n\cdot\nabla F) dS(x) = 0,
\end{align*}
where we used $F$ being harmonic in $\Omega$ and the boundary condition of $\Phi$. Here, $dS(x)$ denotes the induced surface measure on $\p\Omega$. Taking advantage of the orthogonality and the \textit{a priori} bound \eqref{Kt_ineq} for the kinetic energy $K(t)$, we have
\begin{align*}
\int_{\Omega} |\nabla \Phi(t,x)|^2 + |\nabla F(t,x)|^2 dx &= \|u(t,\cdot)\|_{L^2(\Omega)}^2 \le 2K(t) \le 2K(0),
\end{align*}
from which one obtains $\|\nabla F(t,\cdot)\|_{L^2(\Omega)}^2 \le 2K(0)$. Now, since $F \equiv 0$ on $\Gamma_b$, we infer from Poincar\'e inequality  that
\begin{equation}
\label{est:corrH1}
\|F(t,\cdot)\|_{H^1(\Omega)}^2 \le C K(0)
\end{equation}
for some universal constant $C$.
To show the $C^{2}$ bound for $F$, recall that $F$ is harmonic in $\Omega=\mathbb{T}\times[0,1]$ and satisfies the boundary condition $F = 0$ on $\Gamma_b$. Let us oddly extend $F$ to $\tilde \Omega:=\mathbb{T}\times[-1,1]$:
$$
\tilde{F}(x) := \begin{cases}
F(x_1,x_2),& x_2 \ge 0,\\
-F(x_1,-x_2),& x_2 < 0
\end{cases} \quad\text{ for }x=(x_1,x_2)\in\tilde\Omega.
$$
By the Schwarz Reflection Principle for real harmonic functions, $\tilde F:\tilde \Omega\to\mathbb{R}$ is harmonic in $\tilde \Omega$, thus it is also harmonic in the unit disk $B_1 \subset \tilde \Omega$, and obeys the bound \[\|\tilde{F}\|_{H^1(B_1)}^2 \le \|\tilde{F}\|_{H^1(\tilde \Omega)}^2 = 2 \|F\|_{H^1( \Omega)}^2 \leq 2C K(0)\] by \eqref{est:corrH1}. Applying the standard Calder\'on-Zygmund estimate (see \cite[Chapter 2]{Fern_ndez_Real_2022}) and Sobolev embedding, we conclude that
$$
\|\tilde F\|_{C^{2}(B_{1/2})} \lesssim \|\tilde F\|_{H^4(B_{1/2})} \lesssim \|\tilde{F}\|_{H^1(B_1)}\le C \sqrt{K(0)},
$$
which finishes the proof of  \eqref{est_e} after recalling $\nabla e=\nabla (\nabla^\perp \tilde F)$ in $B_{1/2}\cap \Omega$.
\end{proof}

Note that Proposition~\ref{prop_e} only uses the smallness of initial kinetic energy, and we have not used the symmetry of initial velocity yet.  Under additional symmetry assumptions in Lemma~\ref{lem: symmetry}, we arrive at the following:

\begin{prop}
\label{prop:tildeu}
Let the initial data of the system \eqref{eq:fbeulermodel}--\eqref{bdry_b} satisfy the assumptions of both Proposition~\ref{prop_e} and Lemma~\ref{lem: symmetry}.  Then $e_1(t,\cdot)$ is odd in $x_1$ and $e_2(t,\cdot)$ is even in $x_1$ for all times during the lifespan of the solution. Moreover, there exists a universal constant $C$ such that for any $x \in B_{1/2}\cap \Omega$, we have
\begin{align}
\label{est:tildeu}
|e_j(t,x)| \leq C\sqrt{K(0)} |x_j|,\quad j = 1,2.
\end{align}
\end{prop}
\begin{proof}
Using Lemma~\ref{lem: symmetry}, we know $u_1(t,x)$ is odd in $x_1$ and $u_2(t,x)$ is even in $x_1$ for all times in $\calD_t$, thus $\omega(t,x)$ is odd in $x_1$ in $\calD_t$. Since $\Omega\subset \calD_t$ by Proposition~\ref{prop:bdryctrl}, this immediately implies that $\Phi(t,x)$ in \eqref{eq:streamfnc} is also odd in $x_1$ in $\Omega$, due to uniqueness of solution of \eqref{eq:streamfnc}. Using $U(t,x)=\nabla^\perp\Phi(t,x)$,  $U_1(t,x)$ is odd in $x_1$ and $U_2(t,x)$ is even in $x_1$ for all times. Recalling $e = u - U$, we know $e_1(t,x)$ and $e_2(t,x)$ must satisfy the asserted symmetries in $\Omega$.

To show the estimate \eqref{est:tildeu}, let us fix any $x \in B_{1/2}\cap \Omega$. First using the fact that $e_1(t, 0,x_2) = 0$ for $x_2 \in [0,1]$ due to symmetry, we have
$$
|e_1(t,x)| = \left|e_1(t, x_1,x_2) - e_1(t, 0,x_2)\right| \le \|\p_1e_1\|_{L^\infty( B_{1/2}\cap \Omega)} |x_1| \le C \sqrt{K(0)} |x_1|,
$$
where we used \eqref{est_e} in the last inequality. On the other hand, since $e_2(t,x_1,0)= 0$ for $x_1 \in \T$ by \eqref{en}, a similar argument to the above yields \eqref{est:tildeu} with $j = 2$.
\end{proof}

\subsection{Estimating $u$ using integral of $\omega$}
With the error estimate above, we are finally ready to state and prove a pointwise velocity estimate that parallels the lemmas in \cite[Lemma 3.1]{kiselev2014small} and \cite[Lemma 2.1]{zlatovs2015exponential}. In the following, let $\Omega_+ := \T_+ \times [0,1]$. % and $\calD_t^+ := \calD_t \cap \{x_1 > 0\}$.

\begin{prop}
\label{prop:key}
Let the initial data of the system \eqref{eq:fbeulermodel}--\eqref{bdry_b} satisfy the assumptions of both Proposition~\ref{prop_e} and Lemma~\ref{lem: symmetry}. Then for any $x \in B_{1/2}\cap \Omega_+$, the following holds for all time during the lifespan of the solution $[0, T)$:
\begin{equation}
\label{eq:key}
u_j(t,x) = (-1)^j\frac{4}{\pi}\left(\int_{Q(2x)}\frac{y_1y_2}{|y|^4}\omega(t,y)dy + B_j(t,x)\right)x_j,\quad j = 1,2,
\end{equation}
where $Q(x) := [x_1,1] \times [x_2,1]$, and  $B_1$ and $B_2$ satisfies
\begin{equation}
\begin{split}
|B_1(t,x)| \le C_0\left( \|\omega_0\|_{L^\infty(\Omega)}\left(1 + \log\left(1 + \frac{x_2}{x_1}\right)\right) + \sqrt{K(0)}\right),\label{est:r1}\\
|B_2(t,x)| \le C_0\left(\|\omega_0\|_{L^\infty(\Omega)}\left(1 + \log\left(1 + \frac{x_1}{x_2}\right)\right) + \sqrt{K(0)}\right)
\end{split}
\end{equation}
for some universal constant $C_0$.
\end{prop}

\begin{proof}
 Recall that  \eqref{def_e} gives
\[
u(t,x) = U(t,x) + e(t,x)\quad\text{ for }x\in B_{1/2}\cap \Omega_+, t\in[0,T).
\]
In Proposition~\ref{prop:tildeu}, we have already obtained an estimate for the error term $e(t,x)$, namely 
\[ |e_j(t,x)|\leq C\sqrt{K(0)} x_j,\quad \text{ for }j=1,2.
\] (Note that $x_1, x_2\geq 0$ since $x \in B_{1/2}\cap \Omega_+$). Therefore to show \eqref{eq:key}--\eqref{est:r1}, it suffices to prove that
\begin{equation}
\label{eq:key2}
U_j(t,x) = (-1)^j\frac{4}{\pi}\left(\int_{Q(2x)}\frac{y_1y_2}{|y|^4}\omega(t,y)dy + \tilde{B}_j(t,x)\right)x_j,\quad j = 1,2,
\end{equation}
where $\tilde{B}_1, \tilde{B}_2$ satisfy \eqref{est:r1} without the terms $\sqrt{K(0)}$ on the right hand side. 

To show this, let $\tilde\omega$ be the odd-in-$x_2$ extension of $\omega$ from $\Omega=\mathbb{T}\times[0,1]$ to $\mathbb{T}^2$, i.e.
\[
\tilde{\omega}(t,x_1,x_2) := 
\begin{cases}
\omega (t,x_1,x_2 - 2n_2),& x_2 \in (2n_2,1 + 2n_2)\\
-\omega (t,x_1,-(x_2 - 2n_2)),& x_2 \in (-1+ 2n_2, 2n_2)
\end{cases}\quad\text{ for all }n_2\in\mathbb{Z}.
\]
Since $\tilde\omega$ is odd in $x_1$ (by Lemma~\ref{lem: symmetry}) and odd in $x_2$ (by definition of $\tilde\omega$), there exists a unique odd-odd solution $\Psi(t,\cdot)$ to the equation
\begin{equation}\label{def_psi}
\Delta \Psi(t,\cdot) = \tilde{\omega}(t,\cdot)\quad \text{ in }\T^2,
\end{equation}
and $\Psi\in C^{1,\alpha}(\T^2)$ for any $\alpha\in(0,1)$. Note that $\Psi=0$ on both $\T\times\{x_2=0\}$ and $\T\times\{x_2=1\}$ since $\tilde\omega$ is odd about both lines. This implies $\Psi = 0 = \Phi$ on $\partial \Omega$. Combining this with $\Delta \Psi = \omega = \Delta\Phi$ in $\Omega$ leads to $\Psi = \Phi$ in $\Omega$, therefore $U=\nabla^\perp \Phi = \nabla^\perp \Psi$. 

Note that for any $x \in \Omega$, we can express $\Psi(t,x)$ using the Newtonian potential as
\[
\Psi(t,x) = \frac{1}{2\pi} \sum_{n \in \Z^2}\int_{[-1,1]^2} \ln|x-y-2n| \,\tilde{\omega}(t,y) dy
\]
(note that the sum converges since $\tilde\omega$ has mean zero in $[-1,1]^2$), which leads to the following representation of $U$:
\begin{align*}
U(t,x) &= \nabla^\perp\Psi(t,x)= \frac{1}{2\pi}\sum_{n \in \Z^2}\int_{[-1,1]^2}\frac{(x_2 - y_2 - 2n_2, -x_1+y_1 + 2n_1)}{|x-y-2n|^2}\tilde{\omega}(t,y) dy.
\end{align*}
This is exactly the Biot-Savart law for 2D Euler equation in $\mathbb{T}^2$, therefore we can directly use the estimate in \cite[Lemma 2.1]{zlatovs2015exponential} to obtain \eqref{eq:key2}, where $\tilde{B}_1$ and $\tilde{B}_2$ satisfy
\begin{align}
|\tilde{B}_1(t,x)| \le C\|\omega_0\|_{L^\infty(\Omega)}\left(1 + \min\left\{\log\left(1 + \frac{x_2}{x_1}\right), x_2\frac{\|\nabla \omega(t,\cdot)\|_{L^\infty([0,2x_2]^2)}}{\|\omega_0\|_{L^\infty(\Omega)}}\right\}\right),\\
|\tilde{B}_2(t,x)| \le C\|\omega_0\|_{L^\infty(\Omega)}\left(1 + \min\left\{\log\left(1 + \frac{x_1}{x_2}\right), x_1\frac{\|\nabla \omega(t,\cdot)\|_{L^\infty([0,2x_1]^2)}}{\|\omega_0\|_{L^\infty(\Omega)}}\right\}\right)
\end{align}
for some universal constant $C$.
This finishes the proof: note that we can simply drop the second argument in the minimum to arrive at $|\tilde{B}_j(t,x)| \le C\|\omega_0\|_{L^\infty}(1 + \log((1 + \frac{x_{3-j}}{x_j}))$ for $j=1,2$.
\end{proof}

\section{Proof of the main theorem}
Once Proposition~\ref{prop:key} is established, the rest of the proof is largely parallel to the proof of \cite[Theorem 1.1]{kiselev2014small}. However, the situation is slightly more delicate here due to the presence of $\eps$: recall that our initial velocity $u_0=\eps v_0$ depends on $\eps$ since we want to show double-exponential growth can happen for arbitrarily small $\eps\ll 1$. In our proof, we need to construct $v_0$ that is independent of $\eps$, and we need to carefully justify that the double-exponential growth phenomenon happens for any small $\eps$, and quantify the growth rate (which depends on $\eps$).

\begin{proof}[Proof of Theorem~\ref{thm:smallscale}]
Recall that the initial velocity is set as $u_0=\eps v_0$ with $\eps$ sufficiently small, where $v_0\in C^\infty(\calD_0)$ is a fixed velocity field independent of $\eps$. We define $v_0$ as $v_0:=\nabla^\perp \phi$, where $\phi$ solves 
\[
\begin{cases}
\Delta\phi = f & \text{ in } \calD_0,\\
\phi=0 & \text{ on } \partial \calD_0.
\end{cases}
\]
Here $f \in C^\infty(\calD_0)$ is odd in $x_1$, and satisfies $0\leq f\leq 1$ in $\calD_0^+$ and $f(x_1, x_2)=1$ for $x_1 \in [\kappa^{10}, 1-\delta]$, where $\kappa$ and $\delta$ are small universal constants satisfying $0< \kappa < \delta < \frac12$, and they will be fixed momentarily. Since $\|f\|_{L^\infty(\calD_0)}=1$ regardless of $\kappa$ and $\delta$, a standard elliptic estimate gives $\|v_0\|_{L^2(\calD_0)}\leq \|\phi\|_{H^1(\calD_0)}\leq C$ for some universal constant $C$, which implies 
\begin{equation}\label{def_c1}
K(0) = \frac12 \|u_0\|_{L^2(\calD_0)}^2 = \frac{\eps^2}{2} \|v_0\|_{L^2(\calD_0)}^2 \leq C_1 \eps^2
\end{equation} for some universal $C_1$.
Therefore, setting $\eps_0 := (20C_1)^{-1/2}$, we have $K(0)\leq \frac{1}{20}$ for all $\eps\in(0,\eps_0)$. Thus for all $\eps \in (0,\eps_0)$, the initial velocity $u_0$ constructed as above satisfies the small kinetic energy assumption in Proposition~\ref{prop:bdryctrl} and \ref{prop_e} (recall that we set $\sigma=1$ in the assumption of Theorem~\ref{thm:smallscale}). As a result, Proposition~\ref{prop:bdryctrl} implies $\Omega \subset \calD_t$ for all $t\in[0,T)$.
 Due to the odd-in-$x_1$ symmetry of $f$, $u_0$ also satisfies the symmetry assumption in Lemma~\ref{lem: symmetry}.

Since the initial vorticity is $\omega_0 = \eps f$ in $\calD_0$, the set $\{x\in \calD_0^+: \omega_0(x) \neq \eps\}$ has area less than $2\delta$. Using the incompressibility of the flow and the conservation of $\omega$ along the flow map, for any time $t\in[0,T)$, the set $\{x\in \calD_t^+: \omega(t,x) \neq \eps\}$ has area less than $2\delta$. This fact allows us to obtain a lower bound of the integral in \eqref{eq:key}  using a similar argument as \cite[Eq.(3.15)]{kiselev2014small}: for any $t\in[0, T)$ and $x \in B_{\delta}\cap \Omega_+$,
\begin{equation}
\label{int_lower}
\int_{Q(2x)} \frac{y_1y_2}{|y|^4}\omega(y) dy \ge \frac{1}{4} \int_{2\delta}^1 \int_{\pi/6}^{\pi/3}\frac{\omega(r,\theta)}{r}d\theta dr \ge  \frac{\eps}{4}  \int_{4\sqrt{\delta}}^1 \int_{\pi/6}^{\pi/3}\frac{1}{r}d\theta dr  = \frac{\eps \pi}{48} (\log \delta^{-1} - 2\log 4),
\end{equation}
where the first inequality uses the definition of $Q(2x)$ and the fact that $\omega\geq 0$ in $Q(2x)$, and the second inequality uses $|\{x\in \calD_t^+: \omega(t,x) \neq \eps\}| < 2\delta$: in the polar integral of $\omega/r$ if we remove a set with area $2\delta$ closest to the origin from the integral domain $ \{r\in (2\delta,1), \theta\in(\frac{\pi}{6},\frac{\pi}{3})\}$ to reflect the worst-case scenario that minimizes the integral, the remaining set would have inner radius less than $4\sqrt{\delta}$. Also, applying \eqref{def_c1} together with the fact that $\|\omega(t,\cdot)\|_{L^\infty} = \|\omega_0\|_{L^\infty}=\eps$, we can control the terms $B_1$ and $B_2$ in \eqref{est:r1} as
\begin{align}
|B_1(t,x)| \leq C_0  \eps (2+\sqrt{C_1}) =: C_2\eps&\quad \text{ in } B_{\delta}\cap \Omega_+\cap \{0\leq x_2\leq x_1\} \label{B1}\\
|B_2(t,x)| \leq C_0  \eps (2+\sqrt{C_1}) =: C_2\eps &\quad \text{ on } B_{\delta}\cap \Omega_+\cap \{x_2= x_1\}. \label{B2}
\end{align}

From now on, we fix $\delta \in(0,\frac12)$ as a small universal constant such that 
\begin{equation}\label{delta_def1}
\frac{4}{\pi}\left(\frac{ \pi}{48} \left(\log \delta^{-1} - 2\log 4\right)-C_2\right)>1.
\end{equation}

With such definition, combining the estimates \eqref{eq:key} and \eqref{int_lower}--\eqref{B2}, we have 
\begin{align}
-u_1(t,x)&\geq \eps x_1 \quad \text{ in } B_{\delta}\cap \Omega_+\cap \{0\leq x_2\leq x_1\} \label{temp1} \\
u_2(t,x)&\geq \eps x_2 \quad \text{ on } B_{\delta}\cap \Omega_+\cap \{x_2= x_1\}. \label{temp2}
\end{align} In particular, \eqref{temp1} implies the flow map starting from $(\delta,0)$ (denote it by $\eta(t,\delta,0)$) satisfies 
\begin{equation}\label{temp0}
\eta_1(t,\delta,0)\leq \delta e^{-\eps t},\end{equation}
 where we used the fact that $\eta(t,\delta,0)$ stays on the bottom boundary $\Gamma_b$ for all times. Since $\omega(t, \eta(t,\delta,0)) = \omega_0(\delta,0)=\eps$, we know $\|\nabla\omega(t)\|_{L^\infty}$ at least increases exponentially for all times during the lifespan of a solution:
\[
\|\nabla\omega(t,\cdot)\|_{L^\infty(\calD_t)} \geq \frac{|\omega(t, \eta(t,\delta,0))|}{|\eta_1(t,\delta,0)|} \geq \frac{\eps}{ \delta e^{-\eps t}} = \eps \delta^{-1} e^{\eps t}.
\]

To upgrade the exponential growth to double-exponential growth, we follow the same argument as \cite{kiselev2014small}, except that we have to keep track of the dependence on $\eps$ in the growth rate.
For any $t>0$ and $x_1 \in (0,1)$, we define the following two velocities (which is well-defined since $\Omega \subset \calD_t$ for all $t\in[0, T)$):
\begin{equation}
\label{eq:lvel}
\ubar{u}_1(t,x_1) := \min_{(x_1,x_2) \in \Omega_+, x_2 < x_1} u_1(t,x_1,x_2), \quad 
\bar{u}_1(t,x_1) := \max_{(x_1,x_2) \in \Omega_+, x_2 < x_1} u_1(t,x_1,x_2),
\end{equation}
where $\ubar{u}_1$ and $\bar{u}_1$ are locally Lipschitz in $x_1$ during the lifespan of the solution. We then define the functions $a(t)$, $b(t)$ via the ODEs
\begin{align}
a'(t) &= \bar{u}_1(t,a(t)),\quad a(0) = \kappa^{10},\label{eq:aeqn}\\
b'(t) &= \ubar{u}_1(t,b(t)),\quad b(0) = \kappa.\label{eq:beqn}
\end{align}
We also define the following trapezoidal region: for $0 < x_1' < x_1'' < 1$, let
$$
\calO(x_1',x_1'') := \{(x_1,x_2) \in \Omega^+\;:\; x_1'< x_1 < x_1'', ~0\leq x_2 \leq x_1\}.
$$
And we set
$$
\calO_t := \calO(a(t),b(t)).
$$

The choice of our initial data gives $\omega_0\equiv \eps$ in $\calO_0$. We can argue in the same way as \cite[page 1215]{kiselev2014small} that $\omega(t,\cdot)\equiv \eps$ in $\calO_t$:  due to the definition of \eqref{eq:lvel}--\eqref{eq:beqn}, we only need to show $u\cdot(-1,1)>0$ along the diagonal of $\calO_t$. This is true since $u_1<0$ and $u_2>0$ on the diagonal $B_\delta\cap\Omega^+\cap \{x_1=x_2\}$, which follows from \eqref{temp1}--\eqref{temp2}.

Using \eqref{temp1}, we have 
\begin{equation}\label{b}
b(t) \leq \kappa e^{-\eps t}.
\end{equation} To obtain a faster decay for $a(t)$, note that $\log a(t)$ satisfies the differential inequality \[
\begin{split}
\frac{d}{dt}\log a(t) &= 
\frac{\bar{u}_1(t,a(t))}{a(t)} \leq -\frac{4}{\pi}  \left(\int_{Q(2a(t),2a(t))}\frac{y_1 y_2}{|y|^4} \omega(t,y) dy - C_2\eps \right)\\
& \leq  -\frac{4}{\pi} \left(\int_{Q(2a(t),0)}\frac{y_1 y_2}{|y|^4} \omega(t,y) dy - (C_2+C_3)\eps \right) .
\end{split}
\]
where the first inequality follows from \eqref{eq:key} and \eqref{B1}, and the second inequality follows from $\omega\leq\epsilon$ and the fact that for any $a<\frac12$, the integral in the rectangle $\int_{[2a,1]\times[0,2a]} \frac{y_1 y_2}{|y|^4} dy$ is bounded by a universal constant $C_3$. On the other hand, using \eqref{eq:key} and \eqref{B1}, $\log b(t)$ satisfies the differential inequality in the opposite direction:
\[
\frac{d}{dt}\log b(t) =
\frac{\ubar{u}_1(t,b(t))}{b(t)} \geq -\frac{4}{\pi}  \left(\int_{Q(2b(t),0)}\frac{y_1 y_2}{|y|^4} \omega(t,y) dy + C_2\eps \right).
\]

Subtracting them yields the following (where we use that $\calO(2a(t), b(t)) \subset Q(2a(t), 0))\setminus Q(2b(t),0))$):
\begin{equation}\label{temp3}
\frac{d}{dt} \log \frac{b(t)}{a(t)} \geq \frac{4}{\pi} \left( \int_{\calO(2a(t), b(t))} \frac{y_1 y_2}{|y|^4} \omega(t,y) dy - (2C_2+C_3)\eps \right).
\end{equation}
Using $\omega(t,\cdot)\equiv\eps$ in $\calO(2a(t), b(t)) \subset \calO_t$, we can bound the integral in \eqref{temp3} from below as
\[
\int_{\calO(2a(t), b(t))} \frac{y_1 y_2}{|y|^4} \omega(t,y) dy \geq \eps \int_0^{\pi/4} \int_{2a(t)/\cos\theta}^{b(t)/\cos\theta} \frac{\sin 2\theta}{2r} dr d\theta = \frac{\eps}{4} \left(\log\frac{b(t)}{a(t)} - \log2\right),
\]
and plugging it into \eqref{temp3} gives
\[
\frac{d}{dt} \log \frac{b(t)}{a(t)} \geq \eps\left( \frac{1}{\pi} \log \frac{b(t)}{a(t)} - C_4\right),
\]
where $C_4 := \frac{4}{\pi}(\frac{\log 2}{4}+2C_2+C_3)$ is a universal constant. Solving this differential inequality gives
\begin{equation}\label{logtemp}
\log \frac{b(t)}{a(t)} \geq \exp\left(\frac{\eps t}{\pi} \right) \left(\log\frac{b(0)}{a(0)} - \pi C_4\right).
\end{equation}

Since $\log\frac{b(0)}{a(0)} = 9\log\kappa^{-1}$, we can choose $\kappa \in (0, \delta)$ to be a sufficiently small universal constant such that 
$\log\frac{b(0)}{a(0)} - \pi C_4 > 2$. Note that such choice of $\kappa$ guarantees that $2a(t) < b(t)$ for any $t \in [0,T)$. Hence, \eqref{logtemp} implies the following (where we use $b(t)\leq 1$ for all $t$ due to \eqref{b}):
\[
a(t)^{-1} \geq \exp\left(2\exp\left(\frac{\eps t}{\pi}\right)\right) b(t)^{-1} \geq \exp\left(2\exp\left(\frac{\eps t}{\pi}\right)\right).
\]
Finally, using $\omega\equiv\eps$ in $\calO_t$, we have $\omega(t, a(t),0)=\eps$, thus combining it with $\omega(t,0,0)=0$ gives
\[
\|\nabla\omega(t,\cdot)\|_{L^\infty(\calD_t)} \geq \frac{\eps}{a(t)} \geq \eps \exp\left(2\exp\left(\frac{\eps t}{\pi}\right)\right)
\]
for all times during the lifespan of the solution.
\end{proof}

\section{Discussions}

At the end, we discuss some generalizations of Theorem~\ref{thm:smallscale},  and state some open questions.

1. \textbf{Adding gravity to the system}. When a gravity force $-ge_2$ is added to the first equation of \eqref{eq:fbeulermodel}, where $g>0$ and $e_2 = (0,1)^T$, the system becomes the 2D \emph{gravity-capillary} water wave system. Our proof can be easily adapted to this case for $g>0$ and $\sigma>0$. This is because the gravity-capillary water wave system enjoys a similar conserved energy
$
E(t) = K(t) + gP(t) + \sigma L(t),
$
where $P(t) = \int_{\calD_t} x_2 dx$ is the potential energy. It is simple to check that $P(t)\geq P(0)$ for all $t$, since among all sets with the same area as $\mathcal{D}_0$, the set $\mathcal{D}_0$ itself given by \eqref{reference domain} has the lowest potential energy. As a result, the uniform-in-time estimates in Proposition~\ref{prop:bdryctrl} still hold. One can also check that adding gravity still preserves the symmetry in Lemma~\ref{lem: symmetry}. The rest of the proof can be carried out without any changes, and we leave the details to interested readers.
 
 \medskip
2. \textbf{Removing surface tension}. It seems challenging to obtain growth results without surface tension. When $\sigma=0$, the uniform-in-time estimate \eqref{Gamma_t strip} on the free boundary fails, thus the free boundary could potentially get very close to the origin. This difficulty persists even with an additional gravity term -- for gravity water wave without surface tension, if the initial kinetic energy is small, using the conserved energy $
K(t) + gP(t) = K(0) + gP(0)$ one can prove that the free boundary stays close to $\Gamma_0$ in the $L^2$ distance for all times, however, their $L^\infty$ difference can still be large. 

\medskip
 3. \textbf{Different domains}. A natural question is whether the growth result holds for different domains. When the bottom boundary is a graph $\{x_2=g(x_1): x_1\in\mathbb{T}\}$ where $g$ is smooth and even-in-$x_1$, we expect the proof would still hold after some modifications, where the estimate of Biot-Savart law in domains with a symmetry axis by Xu \cite{xu2016fast} could be useful. However, adapting the proof to the infinite-depth case (where there is no bottom boundary) requires substantial new ideas. We also point out that our proof crucially relies on the  periodic-in-$x_1$ setting, and it is an interesting open question to prove similar results for the $x_1\in\mathbb{R}$ case for finite-energy smooth initial data.

\end{document}